\documentclass[12pt]{amsart}
\usepackage{amsmath, amsthm, amssymb, mathrsfs, latexsym, bbm, xypic, epsf, setspace}

\newif\ifpdf
\ifx\pdfoutput\undefined
\pdffalse 
\else
\pdfoutput=1 
\pdftrue
\fi

\ifpdf
\usepackage[pdftex]{graphicx}
\else
\usepackage{graphicx}
\fi

\newtheorem{thm}{Theorem}[section]

\newtheorem{defn}[thm]{Definition}

\newtheorem{cor}[thm]{Corollary}
\newtheorem{lem}[thm]{Lemma}
\newtheorem{rem}[thm]{Remark}

\newtheorem{question}[thm]{Question}

\def\Ind#1#2{#1\setbox0=\hbox{$#1x$}\kern\wd0\hbox to 0pt{\hss$#1\mid$\hss}
\lower.9\ht0\hbox to 0pt{\hss$#1\smile$\hss}\kern\wd0}
\def\Notind#1#2{#1\setbox0=\hbox{$#1x$}\kern\wd0\hbox to 0pt{\mathchardef
\nn="3236\hss$#1\nn$\kern1.4\wd0\hss}\hbox to 0pt{\hss$#1\mid$\hss}\lower.9\ht0
\hbox to 0pt{\hss$#1\smile$\hss}\kern\wd0}

\def\cl{{\rm cl}}

\def\A{\mathcal{A}}
\def\B{\mathcal{B}}
\def\C{\mathcal{C}}
\def\N{\mathbb{N}}
\def\Z{\mathbb{Z}}
\def\Cb{\bar{\mathcal{C}}}

\def\F{\mathcal{F}}
\def\a{\bar{a}}
\def\Ft{\tilde{F}}
\def\P{\mathcal{P}}
\def\H{\mathcal{H}}

\begin{document}

\title[Hrushovski Constructions and Matroids]{Matroid Theory and  Hrushovski's  Predimension Construction.\footnote{Version: 17 May 2011}}

\author{David M. Evans}
\address{School of Mathematics\\
UEA\\
Norwich NR4 7TJ\\
UK}

\email{d.evans@uea.ac.uk}

\begin{abstract}
We give an exposition of some results from matroid theory which characterise the finite pregeometries arising from Hrushovski's predimension construction. As a corollary, we observe that a finite pregeometry which satisfies Hrushovski's flatness condition arises from a predimension. 

\medskip

\end{abstract}

\maketitle

\section{Introduction}
This paper contain some observations about the class of matroids arising from Hrushovski's predimension construction in model theory. We show (Section 2) that these are precisely the strict gammoids: a class of matroids defined by J. H. Mason in the early 1970's. It then follows from a theorem of Mason that a matroid which satisfies Hrushovski's flatness condition can be obtained from a predimension (Corollary \ref{43}). We also give an extension (Theorem \ref{51}) of Mason's theorem which incorporates Hrushovski's notion of self-sufficiency (-- the `correct' notion of embedding in this context).

The paper is written from the viewpoint of a model theorist, though we do not assume any familiarity with the subject. A reader who is more familiar with matroid theory might be interested in the systematic use of the notion of self-sufficiency (which is crucial to the model-theoretic applications) and some of the consequences in later sections, which may be new observations about strict gammoids. However, the paper is largely expository and few, if any, of the results presented here are new. Its purpose is to make the connection between the predimension construction, which is of considerable significance in model theory, and an established part of matroid theory.

We use the term \textit{pregeometry} for a set with a finitary closure operation which satisfies the exchange property, and reserve the term \textit{matroid} for a finite pregeometry.

We begin with an exposition of the basic Hrushovski predimension construction from \cite{Hr}. All of this is well-known: the proofs which we have omitted can be found in \cite{Hr} or \cite{FW} (though most of them are one-line proofs which are easy to reconstruct). We use Oxley's book \cite{Oxley} as our basic reference on matroid theory. The paper \cite{IS} by Ingleton is a very clear survey of results on strict gammoids.

Suppose  $A$ is a set and $R$ is a set of finite, non-empty subsets of $A$ (in fact we can work with a multiset here, allowing repetitions of a subset). For a finite subset $X$ of $A$ we let $R[X] = \{r \in R : r \subseteq X\}$  and, as in \cite{Hr},  define the \textit{predimension} of $X$ to be 
\[ \delta(X) = \vert X \vert - \vert R[X] \vert.\]
Of course, this depends on $(A; R)$, but we suppress this in the notation wherever possible. We are interested in the class $\Cb$ of those $(A; R)$ where $\delta(X) \geq 0$ for all finite $X \subseteq A$. We call $(A; R)$ a \textit{set system} and $\delta$ its associated predimension. If $k \in \N$ we denote by $\Cb_k$ the class of those $(A; R) \in \Cb$ where all the sets in $R$ are of size at most $k$. We write $\C$ and $\C_k$ for the finite members of these classes. For uniformity of notation, we sometimes write $\C_\infty$ instead of  $\C$.

Of course, we can think of the elements of $\Cb$ as structures in a relational language $L$ which has an $n$-ary relation symbol for each $n \in \N$; for $\Cb_k$ we use the sublanguage $L_k$ having relation symbols of arity at most $k$.

\medskip

We recall some further notions and basic results from \cite{Hr}. Suppose $(A; R) \in \Cb$ and $X \subseteq A$ is finite. We write $X \leq A$ to mean that $\delta(X) \leq \delta(X')$ for all finite $X'$ with $X \subseteq X' \subseteq A$, and in this case we say that $X$ is \textit{self-sufficient} in $A$. Using the fact that $\delta$ is submodular (meaning: $\delta(X \cup Y) \leq \delta(X) + \delta(Y) - \delta(X\cap Y)$ for all $X, Y \subseteq A$), it is easy to prove that:
\begin{enumerate}
\item[1.] if $X \leq A$ and $B \subseteq A$ then $X\cap B \leq B$;
\item[2.] if $X \leq Y \leq A$ then $X \leq A$;
\item[3.] if $X, Y \leq A$ then $X\cap Y \leq A$.
\end{enumerate}

Using (1) we can extend the notion of self-sufficiency to arbitrary subsets. Say that $B \leq A$ iff $B\cap C \leq C$ for all finite $C \subseteq A$. Alternatively, $B$ is the union of its finite subsets which are self-sufficient in $A$. The above facts then hold without the assumption that $X, Y$ are finite.

Given $(A; R) \in \Cb$ and (finite) $X \subseteq A$ we define the \text{dimension} $d(X)$ of $X$ to be 
\[d(X) =  \min(\delta(Y) : X \subseteq Y \subseteq A).\]
Note that (by (2, 3) above) there exists a smallest set $Y$ with $X \subseteq Y \leq A$ and for this $Y$ we have $d(X) = \delta(Y)$.

We define the \textit{closure} of a finite subset $X$ to be 
\[\cl(X) = \{y \in A : d(X\cup\{y\}) = d(X)\}.\]
The closure of an arbitrary subset is defined to be the union of the closures of its finite subsets. It can then be shown that:
\begin{enumerate}
\item[4.] $\cl$ is a closure operation on $A$;
\item[5.] $(A, \cl)$ is a pregeometry;
\item[6.] the dimension function of the pregeometry is $d$.
\end{enumerate}
For $\A = (A; R) \in \Cb$ we denote this pregeometry by $PG(\A)$; the associated geometry (on the set of $1$-dimensional closed subsets)  is denoted by $G(\A)$. (Note that we shall reserve the term \textit{matroid} for finite pregeometries.)

If $\A = (A; R) \in \Cb$ and $B \subseteq A$ then we can consider $R\vert B = \{r \in R : r \subseteq B\}$ and look at $\B = (B; R\vert B)$ (-- the induced set system, or substucture,  on $B$; note that we previously denoted $R\vert B$ by $R[B]$). It is clear that this is in $\Cb$, and we can consider $PG(\B)$. In general this will be different from the restriction of $PG(\A)$ to the subset $B$. However, if $B \leq A$ then these will be the same. Indeed, if $X \subseteq B$, let $Y$ be the smallest subset of $A$ with $X \subseteq Y \leq A$. Then by point (3) above, $Y \leq B$ and by (2), $Y$ is the smallest subset of $B$ with $X \subseteq Y \leq B$. Thus $d(X)$ computed in $\B$ is $\delta(Y)$ and this is the same as $d(X)$ computed in $\A$.

\section{Transversals}
The \textit{dual} of a matroid is a matroid on the same set which has as its bases  the complements of the bases of the original matroid. The fact that there is such a matroid is a theorem of Whitney from 1935 (Section 2 of \cite{Oxley} is a convenient reference). It is possible to put together results from the matroid-theoretic literature to see that the matroids  of the form $PG(\A)$ for $\A \in \C$ are exactly the duals of the \textit{transversal matroids} (see 1.6 of \cite{Oxley}). We give a short proof of this result in this section.  By a theorem of Ingleton and Piff \cite{IP},  the duals of transversal matroids  are the \textit{strict gammoids} defined by Mason in \cite{Mason}. So the class of matroids $PG(\C) = \{PG(\A) : \A \in \C\}$ appears in the literature as the class of strict gammoids, or \textit{cotransversal matroids}.

Some connection between transversals and the Hrushovski constructions had already been noted and used in model theory: for example, in \cite{CHKP, DE2}. 

Suppose $A$ is a set and $R$ is a set of finite, non-empty  subsets of $A$. A \textit{transveral} of $(A; R)$ is an injective function $t : R \to A$ with $t(r) \in r$ for all $r \in R$. Abusing terminology, we shall also say that the image $t(R)$ is a transversal of $R$. The following is essentially Lemma 1.5 of \cite{DE2} (in this context, a transversal is what was called an \textit{orientation} of $(A; R)$ in \cite{DE2}). It is a simple consequence of Hall's Marriage Theorem.

\begin{lem} \label{orientation} Suppose $A$ is a finite set and $R$ is a set of non-empty  subsets of $A$. Then $(A; R) \in \C$ iff there is a transversal of $R$. Moreover, if $(A; R) \in \C$ and $B \subseteq A$, then $B \leq A$ iff any transversal of $ R\vert B$ extends to a transversal $t$  of $R$ with the property that  $t(r) \in B \Leftrightarrow r \subseteq B$ iff some transversal of $R\vert B$ extends to a transversal of $R$ in this way.
\end{lem}

\begin{proof}
For the first statement, suppose $(A; R) \in \C$ and let $S \subseteq R$. Let $Y = \bigcup S$. To show that there is a transversal of $R$ (by applying the Marriage Theorem) we need to show that $\vert Y \vert \geq \vert S\vert$. But $\vert S \vert \leq R[Y] \leq \vert Y\vert$, as $\delta(Y) \geq 0$. Conversely if there is a transversal $t : R \to A$ then for any $X \subseteq A$ we have $\vert R[X] \vert = \vert t(R[X]) \vert \leq \vert \bigcup R[X] \vert \leq \vert X \vert$. So $\delta(X) \geq 0$. 

Now suppose $(A; R) \in \C$ and $B \subseteq A$. Then $(B; R\vert B) \in \C$ so there is a transversal $t_0$ of $R\vert B$ (by the first part). Suppose $B\leq A$. To show that $t_0$ extends to a transversal of $R$ of the required form, consider $\{ r\setminus B : r \in R, r\not\subseteq B\}$ (as a multiset). We need to show that if $S$ is a subset of this, then $\vert \bigcup S \vert \geq \vert S \vert$. Let $Y  = B \cup \bigcup S$. Then $\vert S \vert \leq \vert R[Y]\vert - \vert R[B]\vert  = \delta(B) - \delta(Y) + \vert \bigcup S \vert$. As $\delta(B) \leq \delta(Y)$, this gives what we want.

We leave the rest (the converses) as an easy exercise.
\end{proof}

The following (together with Whitney's Theorem) shows that the matroids $PG(\C)$  are the cotransversal matroids.

\begin{thm}\label{cotrans}
Suppose $\A = (A; R) \in \C$. Then $Y \subseteq A$ is a basis of $PG(\A)$ iff $A\setminus Y$ is a transversal of $R$.
\end{thm}

\begin{proof}
Suppose $t : R \to A$ is a transversal with image $A\setminus Y$. Then $R\vert Y$ is empty and so $\delta(Y) = \vert Y\vert$. Also, $t$ extends a transversal of $R \vert Y$ (trivially!), so $Y \leq A$. Now, $\delta(A) = \vert A \vert - \vert R \vert = \vert A \vert - \vert A \setminus Y\vert = \vert Y\vert$. So $A = \cl(Y)$, and $Y$ is a basis for $PG(\A)$. 

Conversely suppose $Y$ is a basis for $A$. So $\vert Y \vert  = d(A)  = d(Y) \leq \delta(Y) \leq \vert Y \vert$. It follows that $R \vert Y$ is empty and $Y \leq A$. By the Lemma, there is a transversal $t : R \to A$ with image in $A\setminus Y$. But $\vert R \vert = \vert A \vert - d(A) = \vert A\vert - \vert Y \vert$. So $t$ has image $A\setminus Y$, and the latter is therefore a transversal of $R$.
\end{proof}

We will use the following in the next section (it is essentially from \cite{Mason}).

\begin{cor}\label{basis} 
Suppose $(A; R) \in \C$ and $t: R \to A$ is a transversal with image $A\setminus Y$. Suppose $F\leq A$. Then $$Z = \{ f \in F\setminus Y: t^{-1}(f) \not\subseteq F\}\cup (F\cap Y)$$ is a basis for $F$.
\end{cor}

\begin{proof}
We know that, because $F \leq A$,  the restriction of the pregeometry $PG(A; R)$ to $F$ is the pregeometry $PG(F; R\vert F)$. Now, $t$ restricted to $R\vert F$ is still a transversal; its image is $F \setminus Z$, so by Theorem \ref{cotrans}, $Z$ is a basis of $PG(F; R\vert F)$. 
\end{proof}

\begin{rem}\rm If $(A; R')$ is a finite set system, then the  associated  \textit{transversal matroid} has as its independent subsets the (images of) partial transversals of $(A; R')$ (see \cite{Oxley}, 1.6). There is no requirement here that $(A; R') \in \C$: however, 2.4.1 of \cite{Oxley} shows that there is a subset $R \subseteq R'$ with  $(A; R) \in \C$  such that $(A;R)$ has the same associated transversal matroid as $(A; R')$. 

In fact, the definition of transversal matroid given in \cite{Oxley} is apparently more general: one works with partial transversals of $(A; R')$ where $R'$ is a multiset. Of course, we can adapt the definition of predimension to accommodate this and the above results still hold. However, it is fairly easy to show that if $R'$ is a multiset of subsets of $A$ and $(A; R') \in \C$, then there is a set $R$ of subsets of $A$ such that $(A; R) \in \C$ and $PG(A; R') = PG(A; R)$.
\end{rem}

\begin{rem}\rm For completeness, we give the definition of strict gammoids from \cite{Mason}. Suppose $\Gamma = (A; D)$ is a directed graph (without loops) with vertices $A$ and directed edges $D$. Suppose $B$ is a subset of $A$. In the \textit{strict gammoid} on $A$ determined by these, a subset $C \subseteq A$ in independent iff it is \textit{linked} to a subset of $B$: this means that there is a set of disjoint directed paths with the vertices in $C$ as initial nodes and whose terminal nodes are in $B$. 

Suppose $(A; R) \in \C$ and $t : R \to A$ is a transversal with image $A\setminus B$. Define a directed graph $\Gamma$ on $A$ with directed edges $\{ (t(r), c) : r \in R, \,\,  c \in r, \,\, c\neq t(r)\}$. Then it can be shown that $PG(A; R)$ is the strict gammoid given by $\Gamma$ and $B$.
\end{rem}

\section{The $\alpha$-function}
In this sections we give Mason's  characterization of the strict gammoids (-- that is, the matroids $PG(\C)$) from \cite{Mason}. A slight modification to the proof allows us to give a charcterization of the matroids $PG(\C_k)$. Almost all of the following is taken from \cite{Mason}, but we have adapted it to deal directly with $PG(\C)$ rather than going via linkages in directed graphs (as in the original presentation).

\begin{defn}\rm
(1)  Suppose $(A, \cl)$ is a matroid with dimension function $d$. If $F$ is a closed subset of  $A$ (sometimes called  a \textit{flat} in $A$), then we write $F \sqsubseteq A$ to indicate this. More generally, and slightly ambiguously, if $X \subseteq A$ we write $F \sqsubseteq X$ to mean that $F$ is a closed subset of  $A$ and $F \subseteq X$. With strict containment, we write $F \sqsubset X$. \newline
(2)  The \textit{$\alpha$-function} of $(A, \cl)$ is defined (inductively on dimension) on unions of closed sets in $A$ by the rule:
\[\alpha(X) = \vert X \vert - d(X) - \sum_{G \sqsubset X} \alpha(G)\]
when $X$ is a union of closed sets.
\end{defn}

When $X$ is a closed set $F$ we can rewrite this as:
\[ \sum_{G \sqsubseteq F} \alpha(G) = \vert F \vert - d(F)\]
and this formula can be inverted using the M\"obius function of the lattice of closed stes. Note however that in what follows it will be important to consider $\alpha(X)$ when $X$ is \textit{not} a closed set.

Clearly $\alpha(\cl(\emptyset)) = \vert \cl(\emptyset)\vert$ and if $F$ is a point (that is, a closed set of dimension 1) then $\alpha(F) = \vert F \setminus \cl(\emptyset)\vert -1$. We can normalise by passing to the geometry $(\hat{A}, \cl)$ of $A$ and considering its $\alpha$-function. Thus, in the geometry we have $\alpha(\emptyset) = 0$, $\alpha(p) =0$ for a point $p$, and $\alpha(\ell) = \vert \ell\vert -2$ for a line $l$. A straightforward induction on  $d(X)$ shows that if $X$ is a closed set and $d(X) \geq 2$ then $\alpha(X) = \alpha(\hat{X})$, where $\hat{X}$ is the image of $X$ in the geometry.

\begin{defn}\rm  Suppose $(A, \cl)$ is a matroid and let $\F$ denote the set of closed sets in $(A, \cl)$. Suppose $\gamma: \F \to \Z$ is such that $\gamma(F) \geq 0$ for all $F \in \F$. By a  \textit{$\gamma$-transversal} of $\F$ we mean a collection $(X_F : F \in \F)$ of pairwise disjoint sets such that  $X_F$ is a subset of $F$ of size $\gamma(F)$, for each $F \in \F$.
\end{defn}

\begin{thm}[Mason, \cite{Mason}] \label{32} Suppose $(A, \cl)$ is a matroid. The following are equivalent:
\begin{enumerate}
\item Whenever $X \subseteq A$ is a union of closed sets, then $\alpha(X) \geq 0$.
\item There is a set $R$ of non-empty subsets of $A$ such that $\A = (A; R) \in \C$ and $PG(\A) = (A, \cl)$.
\item There is an $\alpha$-transversal of the set of closed subsets of $A$.
\end{enumerate}
Moreover, we can take $R$ in (2) to be a set of subsets of size at most $k$ iff $\alpha(F) = 0$ for all closed sets $F$ of dimension at least $k$.\end{thm}

\begin{proof}
\textbf{$(2) \Rightarrow (1)$:\/} Suppose (2) holds and let $t : R \to A$ be a transversal (Lemma \ref{orientation}) with image $A \setminus Y$. For a flat $F$ define $\beta(F) =\vert\{ x \in F\setminus Y : \cl(t^{-1}(x)) = F\}\vert$. We prove by induction on $d(F)$ that $\beta(F) = \alpha(F)$. If $F = \cl(\emptyset)$ then a basis for $F$ is empty. So by Corollary \ref{basis}, $F \subseteq A\setminus Y$ and $t^{-1}(x) \subseteq F$ for all $x \in F$: thus $\beta(F) = \vert \cl(\emptyset)\vert = \alpha(F)$. In general suppose we have the claim for $G \sqsubset F$. Note that $F$ is the disjoint union of subsets:
\begin{enumerate}
\item[(i)] $\{ x \in F\setminus Y : \cl(t^{-1}(x)) \not\subseteq F\}\cup (F \cap Y)$;
\item[(ii)]  $\{ x \in F \setminus Y : \cl(t^{-1}(x)) = F\}$;
\item[(iii)] $\{ x \in F \setminus Y : \cl(t^{-1}(x)) \sqsubset F\}$.
\end{enumerate}
The first of these has size $d(F)$ (by Corollary \ref{basis}); the second has size $\beta(F)$. By the inductive hypothesis, the third has size $\sum_{G \sqsubset F} \alpha(G)$. Thus 
$\beta(F) = \vert F\vert - d(F) - \sum_{G\sqsubset F} \alpha(G) = \alpha(F)$, as required. So of course, this shows that $\alpha(F) \geq 0$ if $F$ is a flat. Now suppose that $X$ is a union of flats (but is not a flat). Then $\alpha(X) = \vert X \vert -d(X) - \sum_{F \sqsubseteq X} \beta(F)$. The sum in this is equal to $\vert\{ x \in X \setminus Y: \cl(t^{-1}(x)) \sqsubseteq X\}\vert$ and this is at most $\vert R[X]\vert$. So $\alpha(X) \geq \vert X \vert - \vert R[X]\vert - d(X)= \delta(X) - d(X) \geq 0$. 

\medskip

\textbf{$(1) \Rightarrow (3):$\/} Suppose (1) holds. Let $\F$ denote the set of closed sets in $(A, \cl)$. We first show that there is an $\alpha$-transversal of $\F$. By a generalization of Hall's Marriage Theorem (quoted in \cite{Mason} as due to Welsh, but attributed to Halmos and Vaughan in \cite{Mir}), it will suffice to prove that for distinct $F_1, \ldots, F_r \in \F$ we have $\vert \bigcup_{i\leq r} F_i\vert \geq \sum_{i\leq r} \alpha(F_i)$. If the union is not one of the $F_i$, then 
\[\vert \bigcup _i F_i \vert = \alpha(\bigcup_i F_i) + d(\bigcup_i F_i)+ \sum_{F \sqsubset \bigcup_i  F_i} \alpha(F)\geq \sum_i \alpha(F_i)\]
using $\alpha \geq 0$. If $\bigcup_i F_i = F_r$ then a similar argument gives what we want.

\textbf{$(3) \Rightarrow (2):$\/}   As before, let $\F$ denote the set of closed subsets in $(A, \cl)$ and suppose $(X_F : F\in \F)$ is an $\alpha$-transversal of $\F$ (so, implicit in  this is that $\alpha(F) \geq 0$ for all $F \in \F$).

Now let $Y = A \setminus \bigcup_{F \in \F} X_F$. If $F \in \F$ and $a \in X_F$ then let $S(F) = F \setminus \bigcup_{G\sqsubseteq F} X_G$ and $R_a = \{a\}\cup S(F)$. Note that these are distinct (for example, this follows from Claim 1 below). We let $R = \{R_a : a \in X_F, \, F \in \F\}$. In a series of claims we show that $(A; R)$ satisfies (2). Let $\delta$ denote the predimension coming from $(A;R)$. Of course, we have a transversal $t: R \to A$ (given by $t(R_a) = a$), so $(A; R) \in \C$. 

\smallskip

\noindent\textit{Claim 1:\/} $S(F)$ is a basis for $F$. \newline
Note that $\vert S(F)\vert  = \vert F \vert - \sum_{G \sqsubseteq F} \alpha(G) = d(F)$. So it will suffice to show that $S(F)$ is not contained in $G$ when $G \sqsubset F$. By definition of $S(G)$ we have  $S(F) \cap G \subseteq S(G)$. As $S(G)$ has size $d(G) < d(F)$ we therefore cannot have $S(F) \subseteq G$. 

\smallskip

\noindent\textit{Claim 2:\/} If $F$ is a closed set then $\delta(F) = d(F)$. \newline
First suppose $a \in X_G$ and $R_a \subseteq F$. Then $G = \cl(R_a) \subseteq F$. On the other hand, if $G \sqsubseteq F$ and $a \in X_G$ then clearly $R_a \subseteq F$. Thus $\vert R[F] \vert  = \sum_{G \sqsubseteq F} \alpha(G) = \vert F \vert - d(F)$. So $d(F) = \delta(F)$.

\smallskip 

\noindent\textit{Claim 3:\/} If $G$ is a closed set and $S(G) \subseteq C \leq G$, then $\delta(C) = \delta(G)$.\newline
If $a \in G\setminus C$ then $a \in X_H$ for some $H\sqsubseteq G$, thus $R_a \in R[G]\setminus R[C]$. Thus $\vert G \vert -\vert C\vert \leq \vert R[G] \vert -\vert R[C]\vert$. So $\delta(G) \leq \delta(C)$. As $C \leq G$ we get equality here.

\smallskip

\noindent\textit{Claim 4:\/} If $X \subseteq A$ then $\delta(X) \geq d(X)$.\newline
Let $F = \cl(X)$. We can assume that $X \leq F$ (if there is $Y$ with $X \subseteq Y \subseteq F$ and $\delta(Y) < \delta(X)$ then $\cl(Y) = F$ and it will suffice to prove $\delta(Y) \geq d(F)$). Suppose $G \sqsubseteq F$ and $S(G) \subseteq X$. We show that $\delta(X\cup G) = \delta(X)$. Let $C = X \cap G$. By Claim 3, we have $\delta(C) = \delta(G)$. By submodularity, we then have $\delta(X \cup G) \leq \delta(X)+\delta(G) - \delta(C) = \delta(X)$. Again, the fact that $X \leq F$ gives the equality. We may therefore assume, without changing the value of $\delta(X)$,  that  if $S(G) \subseteq X$ then $G \subseteq X$. It then follows that  $\vert R[X]\vert = \sum_{G: S(G) \sqsubseteq X} \vert X_G\vert = \sum_{G: S(G) \sqsubseteq X} \alpha(G)$, and moreover the latter is equal to  $\sum_{G \sqsubseteq X} \alpha(G)$. But this is $\vert X \vert - d(X)$: so $\delta(X) = d(X)$, as required.

This finishes the proof that $PG(A; R) = (A, \cl)$: if $X \subseteq A$ then $d(X) = \min(\delta(Y) : X \subseteq Y \subseteq A)$, by claims 3 and 4.

For the `moreover' part note that from the first part of the proof, if all sets in $R$ are of size at most $k$, then $\beta(F) = 0$ whenever $d(F) \geq k$ (because if $r \in R$ then $d(\cl(r)) \leq \delta(r) \leq \vert r \vert -1$). Thus the same is true for $\alpha(F)$. Conversely, if $\alpha(F) = 0$ whenever $d(F) \geq k$ then the construction in the step $(3) \Rightarrow (2)$ of the proof has all of the $R_a$ of size at most $k$.
\end{proof}

We refer to a set $R$ such that $PG(A; R) = (A, \cl)$ as a \textit{presentation} of the pregeometry $(A, \cl)$.

\begin{rem}\rm
The steps $(2)\Rightarrow (1) \Rightarrow (3)$ in the above are essentially as in  Mason's paper. The construction of the presentation in $(3) \Rightarrow (2)$ is different from the corresponding step in Mason's argument (and results in a presentation with smaller sets, giving the `Moreover' part of the result). In fact, the presentation constructed here is a minimal presentation, rather than the maximal presentation given in Mason's paper (cf. \cite{IS}, Section 3). 
\end{rem}

\section{Flatness}

In Section 4.2 of \cite{Hr}, Hrushovski introduces the notion of a (not necessarily finite) pregeometry being \textit{flat}.

\begin{defn}\label{41} \rm
Suppose $(A, \cl)$ is a pregeometry with dimension function $d$ and $\F = \{F_i : i \in I\}$ is a non-empty, finite set of (distinct) finite-dimensional closed sets. If $\emptyset \neq S \subseteq I$ let $F_S = \bigcap_{i \in S} F_i$ and let $F_{\emptyset} = \bigcup_{i \in I} F_i$. Let 
$$\Delta(\F) = \sum_{S \subseteq I} (-1)^{\vert S\vert} d(F_S).$$
We say that $(A, \cl)$ is \textit{flat} if $\Delta(\F) \leq 0$ for all such $\F$.
\end{defn}

Lemma 15 of \cite{Hr} (or rather, its proof) shows that if $(A; R) \in \Cb$,  then $PG(A; R)$ is flat. We prove the converse of this for finite pregeometries. First we connect $\Delta$ with the $\alpha$-function. (Similar results, but in the dual context of transversal matroids appear in \cite{IS}, Section 3.)

\begin{lem}\label{42}
Suppose $(A, \cl)$ is a finite pregeometry (with dimension function $d$). Let $X \subseteq A$ be a union of closed sets (with $d(X) \geq 2$) and define $\F(X) = \{ F : F \mbox{ is a closed set in $A$ and } F \subset X\}$. Then 
\[ \alpha(X) = - \Delta(\F(X)).\]
\end{lem}

\noindent\textit{Remark:\/} The strict containment in the definition of $\F(X)$ is needed here. If $X$ is closed then one computes easily that $\Delta(\F(X)\cup \{X\}) = 0$. 

\medskip

\begin{proof}
Write $\F(X)$ as $(F_i : i \in I)$ as in the definition of flatness. Note that $X = \bigcup \F(X)$, so $F_{\emptyset} = X$. Then
\[\alpha(X) = \vert F_\emptyset \vert - d(F_\emptyset) - \sum_{i \in I} \alpha(F_i)\]
and
for $\emptyset \neq S \subseteq I$
\[\alpha(F_S) = \vert F_S\vert - d(F_S) - \sum_{H \sqsubset F_S} \alpha(H).\]
Note that by inclusion-exclusion
\[\sum_{S \subseteq I} (-1)^{\vert S\vert} \vert F_S\vert = 0.\]
Thus
\[\alpha(X)+ \sum_{S\neq \emptyset} (-1)^{\vert S\vert} \alpha(F_S) = \]\[0 - \sum_S (-1)^{\vert S\vert}d(F_S) - \sum_{i \in I} \alpha(F_i) - \sum_{S\neq \emptyset}(\sum_{H \sqsubset F_S} (-1)^{\vert S \vert} \alpha(H)).\]
So 
\[\alpha(X) + \Delta(\F(X)) = - ( \sum_{i \in I }\alpha(F_i) + \sum_{S\neq \emptyset}(\sum_{H \sqsubseteq F_S} (-1)^{\vert S \vert} \alpha(H))).\]
Let $H \in \F$: so $H = F_i$ for some $i \in I$. The `contribution' of this to the second sum on the right-hand side is

\[\alpha(H) \sum_{S: S\neq \emptyset,  \, H \sqsubseteq F_S} (-1)^{\vert S\vert}.\]
Let $I_H = \{j \in I : H \sqsubseteq F_j\}$. So the above summation has $S$ ranging over the non-empty subsets of $I_H$. As $I_H \neq \emptyset$ we therefore have $\sum_{\emptyset \neq S \subseteq I_S} (-1)^{\vert S \vert} = -1$. So $\alpha(H) \sum_{S\neq \emptyset,  \, H \sqsubseteq F_S} (-1)^{\vert S\vert}  = \alpha(H) = \alpha(F_i)$. Thus $\alpha(X) + \Delta(\F(X)) = 0$, as required.
\end{proof}

\begin{cor}\label{43}
Suppose $(A, \cl)$ is a finite pregeometry. Then the following are equivalent (where the notation $\F(X)$ is as in the previous lemma).
\begin{enumerate}
\item There is a set $R$ of non-empty subsets of $A$ such that $(A; R) \in \C$ and $PG(A; R) = (A, \cl)$.
\item For all subsets $X$ of $A$ which are unions of closed sets, we have $\alpha(X) \geq 0$.
\item For all subsets $X$ of $A$ which are unions of closed sets, we have $\Delta(\F(X)) \leq 0$.
\item $(A, \cl)$ is flat.
\end{enumerate}
Moreover, we can take $R$ in (1) to consist of sets of size at most $k$ iff $\Delta(\F(X)) = 0$ whenever $X$ is a closed set of dimension at least $k$.
\end{cor}

\begin{proof}
Equivalence of (1), (2) and (3) follows from Theorem \ref{32} and Lemma \ref{42}. Clearly (4) implies (3), and by Hrushovski's result, (1) implies (4).

The `moreover' part is just from Theorem \ref{32}.
\end{proof}

\begin{rem}\rm \label{44}
If $(A; R) \in \C$ and $F_1,\ldots,F_r$ are closed sets in the associated pregeometry, it follows from the definition that 
\[- \Delta(F_1,\ldots,F_r) = \delta(\bigcup_{i=1}^r F_i) - d(\bigcup_{i=1}^r F_i) + \rho(F_1,\ldots, F_r)\]
where $\rho(F_1,\ldots, F_r)$ is the number of relations on $\bigcup_{i=1}^r F_i$ (that is, subsets of this which are in $R$) which are not contained in one of the $F_i$. This proves flatness of $PG(A; R)$. Note also that $\Delta(F_1,\ldots,F_r) = 0$ iff $\bigcup_{i=1}^r F_i \leq A$ and  any relation on $\bigcup_{i=1}^r F_i$ is contained in one of the $F_i$.
\end{rem}

\section{Self-sufficient embedding}

If $\A = (A, \cl)$ is a pregeometry and $C \subseteq A$, then we can consider the restriction to $C$: the pregeometry $(C, \cl_C)$ on $C$ with closure $\cl_C(X) = \cl(X)\cap C$, for $X \subseteq C$ (in matroid terms, this is the deletion of $A\setminus C$). If $A$ is finite, we can also consider the $\alpha$-function $\alpha_C$ on $(C, \cl_C)$, and in general there is no reason to expect any connection between $\alpha_A$ and $\alpha_C$.

The following result gives a characterization in terms of $\alpha_C$ and $\alpha_A$  of the subsets $C \subseteq A$ for which there is a presentation $(A; R)$ of $\A$ such that  $C \leq (A;R)$. This can be seen as a generalization of Theorem \ref{32} (which is the case $C = \emptyset$).

\begin{thm}\label{51}
Suppose $\A = (A, \cl) \in PG(\C)$ and $C \subseteq A$. The following are equivalent:
\begin{enumerate}
\item There is a presentation $(A; R)$ of $\A$ such that $C \leq A$.
\item There is a presentation $(A; R)$ of $\A$ such that $(C; R\vert C)$ is a presentation of $(C, \cl_C)$.
\item  There is an $\alpha_A$-transversal $(X_F: F \sqsubseteq \A)$ of the closed sets of $\A$ such that if $H$ is a closed set of $(C, \cl_C)$, then $X_{\cl(H)}\cap C$ has size $\alpha_C(H)$.
\item $\alpha_C(Y) \geq 0$ whenever $Y$ is a union of closed sets of $(C,\cl_C)$, and if $X$ is a union of closed sets in $A$, then
\[\alpha_A(X) \geq \alpha_C(X\cap C) + \sum \{\alpha_C(J) : J\sqsubset_C X\cap C \mbox{ and } \cl_A(J)\not\subset X\}.\]
\end{enumerate}
\end{thm}

\noindent\textit{Remark:\/} In (4), $J \sqsubset_C X\cap C$ means that $J \subset X\cap C$ and $\cl_C(J) = J$. Note that in the case where $X$ is a closed subset of $A$, we have $\cl_A(J) \subset X$, so the condition reduces to $\alpha_A(X) \geq \alpha_C(X\cap C)$.

Before we begin the proof, we note the following simple result:

\begin{lem}\label{50} Let $(A; R) \in \C$ and $C \leq A$. Then for any $X \subseteq A$ we have 
$\delta(X) - d(X) \geq \delta(X\cap C) - d(X\cap C)$.
\end{lem}

\begin{proof} Let $Z = \cl(X\cap C)$. By point (1) in Section 1 we have $X\cap C \leq X$. So $\delta(X\cap C) \leq \delta(Z\cap C)$. As $X, Z \subseteq \cl(X)$ we have $d(X) \leq \delta(X\cup Z)$ and, using other points from Section 1, the following calculation:
\begin{multline*}
d(X) \leq \delta(X\cup Z) \leq \delta(X)+\delta(Z)- \delta(X\cap Z) = \\
\delta(X) + d(X\cap C) - \delta(X\cap Z) \leq 
\delta(X) - d(X\cap C) - \delta(X\cap C).
\end{multline*}
Rearranging gives what we require.
\end{proof}

\begin{proof}[Proof of 5.1] The proof follows that of Theorem \ref{32}, so we only sketch some of the arguments.

That (1) implies (2) is given in Section 1. For the converse, we show that $C \leq A$ in the presentation from (2). We know that $d(C)$ is the same whether computed in $(C; R\vert C)$ or $\A$. Thus if $C \subseteq B \subseteq A$ we have 
$\delta(C) = d(C) \leq d(B) \leq \delta(B)$,
as required.

\medskip

To show that (1) implies (4), we take a transversal $t$ of $R$ with the property that $t(r) \in C \Leftrightarrow r \subseteq C$ (using Lemma \ref{orientation}). Let $Y$ be the complement of the image of $t$ in $A$. So $t \vert (R\vert C)$ is a transversal of $R\vert C$ with image $C\setminus Y$. 

Let $X$ be a union of closed sets in $A$. As in $(2) \Rightarrow (1)$ of Theorem \ref{32}, we have:
\[\alpha(X) = \vert X \vert - d(X) - \vert\{x\in X\setminus Y : \cl(t^{-1}(x)) \subset X\}\vert,\]
and similarly
\[\alpha_C(X\cap C) = \vert X \cap C\vert - d(X\cap C) -\vert \{x\in (X\cap C)\setminus Y: \cl_C(t^{-1}(x)) \subset X\cap C\}\vert.\]
If $x \in X\setminus Y$  and $\cl(t^{-1}(x)) \subset X$, then either $x \not\in C$, so $t^{-1}(x) \in R[X]\setminus R[C]$; or $x \in C$, so $x \in X \cap C$ and $\cl_C(t^{-1}(x)) \subseteq X\cap C$. Thus:
\begin{multline*}\vert\{x \in X \setminus Y: \cl(t^{-1}(x)) \subset X\}\vert \leq \\
\vert R[X]\setminus R[C] \vert +  \vert \{x \in X\setminus Y: \cl(t^{-1}(x)) \subset X, \,\, \cl_C(t^{-1}(x)) \subseteq X \cap C\}\vert.\end{multline*}
The second term in the sum here is:
\begin{multline*}
\vert\{x \in X\cap C\setminus Y : \cl_C(t^{-1}(x)) \subseteq X\cap C, \,\, \cl(t^{-1}(x)) \subset X\}\vert 
+\\ \vert\{x \in X\cap C\setminus Y : \cl_C(t^{-1}(x)) \subseteq X\cap C, \,\, \cl(t^{-1}(x)) \not\subset X\}\vert .
\end{multline*}
The second summand is equal to 
\[\sum \{\alpha_C(J) : J\sqsubset_C X\cap C \mbox{ and } \cl_A(J)\not\subset X\}.\]
If $X\cap C$ is not closed in $(C,\cl_C)$, the first summand is equal to:
\[\vert X \cap C\vert - d(X\cap C) - \alpha_C(X\cap C),\]
as above.

Putting these together, we obtain
\begin{multline*} \alpha(X) \geq \alpha_C(X\cap C) + \delta(X) - \delta(X\cap C) - d(X) + d(X\cap C) +\\
\sum \{\alpha_C(J) : J\sqsubset_C X\cap C \mbox{ and } \cl_A(J)\not\subset X\}.
\end{multline*}

It can be checked that the same inequality also holds if $X\cap C$ is closed in $(C,\cl_C)$. Then (4) follows from  Lemma \ref{50}.

Now suppose that (4) holds, and deduce (3). By the first assumption in (4) and Theorem \ref{32} we can assume that we have an $\alpha_C$-transversal of the closed sets in $(C,\cl_C)$. We show that this can be extended to an $\alpha$-transversal of the closed sets of $A$ using only elements of $A\setminus C$. So suppose $F_1,\ldots, F_r$ are clsoed subsets of $A$ and $\cl(F_i\cap C) = F_i$ iff $i \leq s$. As in the proof of $(1) \Rightarrow (2)$ of Theorem \ref{32}, it will be enough to show that if $X = \bigcup_{i \leq r} F_i$ then 
\[\vert X \setminus C\vert \geq \sum_{i\leq r} \alpha(F_i) - \sum_{i\leq s} \alpha_C(F_i \cap C).\]
By definition of $\alpha$ and $\alpha_C$
\begin{multline*} 
\vert X\setminus C\vert = \alpha(X) - \alpha_C(X\cap C) + d(X) - d(X\cap C) + \\
\sum_{G \sqsubset X} \alpha(G) - \sum_{H \sqsubset_C X\cap C} \alpha_C(H).
\end{multline*}
By (4) this is 
\begin{multline*}
\geq  \sum \{\alpha_C(J) : J\sqsubset_C X\cap C \mbox{ and } \cl_A(J)\not\subset X\} + \\
\sum_{G \sqsubset X} \alpha(G) - \sum_{H\sqsubset_CX\cap C}\alpha_C(H)\end{multline*}
\begin{multline*}
= \sum_{G \sqsubset X} \alpha(G) - \sum_{H\sqsubset_C X\cap C, \cl(H) \subset X} \alpha_C(H) 
\geq \sum_{i\leq r} \alpha(F_i) - \sum_{i\leq s}\alpha_C(F_i\cap C),
\end{multline*}
because $\alpha(\cl(H)) - \alpha_C(H) \geq 0$ if $H \sqsubseteq_C C$, using (4) as in the remark afer the statement of the theorem.

\medskip

Finally, we suppose that  (3) holds and deduce (2). Using the given $\alpha_A$-transversal, perform the construction of $R$ given in the step $(3) \Rightarrow (2)$ of the  proof of Theorem \ref{32}. We want to show that $(C, R\vert C)$ is a presentation for $(C,\cl_C)$. 

Recall some of the details of the construction of $R$. If $F$ is a closed set in $\A$ let $S(F) = F \setminus\bigcup_{G \sqsubseteq F}X_G$ and for $a \in X_F$ let $R_a = \{a\}\cup S(F)$. Then $R$ is the set of these. Similarly if $H$ is a closed set in $(C, \cl_C)$, let $X^C_H = X_{\cl(H)} \cap C$ and $S^C(H) = H \setminus \bigcup_{G\sqsubseteq H} X^C_G$. For $a \in X^C_G$ let $R^C_a = \{a\}\cup S^C(H)$ and let $R^C$ be the set of all these as $H$ ranges over the closed sets of $(C,\cl_C)$. By the proof of Theorem \ref{32}, $(C; R^C)$ is a presentation of $(C, \cl_C)$. So it will suffice to prove that $R\vert C = R^C$.

\smallskip

\noindent\textit{Claim:\/} If $H$ is a closed set in $(C,\cl_C)$, then $S^C(H) = S(\cl(H))$.\newline
Note that 
\[S^C(H) = (\cl(H) \setminus \bigcup_{G \sqsubseteq H} X_{\cl(G)})\cap C \subseteq S(\cl(H)) \cap C.\]
But (from Claim 1 in the proof of Theorem \ref{32}) we have 
\[\vert S^C(H)\vert = d(H) = d(\cl(H)) = \vert S(\cl(H))\vert,\]
hence the claim holds.

\smallskip

Using the definitions of $R^C$ and $R$, this shows that  $R^C \subseteq R \vert C$. To get the equality, note that as 
 $(C; R^C)$ is a presentation of $(C, \cl_C)$ we have $d(C) = \vert C \vert - \vert R^C\vert$. On the other hand, $(A; R)$ is a presentation of $A$ so $d(C) \leq \delta(C) = \vert C \vert - \vert (R \vert C)\vert$. Thus $\vert (R\vert C)\vert \leq \vert R^C\vert$ and the equality follows.
\end{proof}

This enables us to give a different proof of the finite case of Theorem 4.3 of \cite{MFDE1}.

\begin{cor}
Suppose $m < n$ and $(A ; R) \in \C$ is such that there exists $C \leq A$ with $\vert C \vert = n$, $d(C) = n-1$ and every $(n-1)$-subset of $C$ is independent. If $(B; R') \in \C_m$, then $PG(A; R)$ is not isomorphic to $PG(B; R')$. 
\end{cor}

\begin{proof}
It is easy to see that $\alpha$ of any independent set is 0. Thus $\alpha_C(C) = \vert C \vert - d(C) = 1$. So, by the Remarks after \ref{51}, $\alpha_A(\cl(C)) \geq 1$. But if $F$ is a closed set in $(B; R') \in \C_m$ of dimension at least $m$, then $\alpha_B(F) = 0$, by Theorem \ref{32}.
\end{proof}

Recall that if $k \in \{3,4,\ldots, \infty\}$ then $\C_k$ is the class of finite set systems $(A; R)$ with $\emptyset \leq A$ and the sets in $R$ of size at most $k$. We denote by $\P_k = PG(\C_k)$ the class of pregeometries $PG(A; R)$ with $(A; R) \in \C_k$ and for $B \subseteq \A \in PG(\C_k)$ we write $B \unlhd_k A$ to mean that there is a presentation $(A; R) \in \C_k$ of $\A$ with $B \leq (A; R)$. It can be shown that $\unlhd_k$ is transitive (\cite{MFDE1}, 6.3). We usually write $(\P, \unlhd)$ instead of $(\P_{\infty}, \unlhd_\infty)$. 

Thus, in this notation, Theorem \ref{51} gives a characterizaton of $\P_k$ and the relation $\unlhd_k$ in terms of the $\alpha$-function. However, the following shows that the relation $\unlhd_k$ is redundant.

\begin{cor}
Suppose $\A = (A, \cl) \in PG(\C_k)$ and $C \unlhd_\infty \A$. Then $C \unlhd_k \A$.
\end{cor}

\begin{proof} By assumption, there is a presentation $(A; R') \in \C$ of $\A$ with $C \leq (A; R')$. We want to show that there is such a presentation in $\C_k$. By (1 $\Rightarrow$ 3) of Theorem  \ref{51}, condition (3) of Theorem \ref{51} holds. As $\A \in PG(\C_k)$, we have $\alpha_A(F) = 0$ whenever $F$ is a closed subset of $\A$ of dimension at least $k$ (as in Theorem \ref{32}). Then the presentation $(A; R)$ of $\A$ built using the $\alpha_A$-transversal in step (3 $\Rightarrow$ 2) of \ref{51} is in $\C_k$ and has $C \leq (A; R)$.
\end{proof}

Note that although the $\alpha$-function does not make sense for infinite pregeometries, the definition of flatness (Definition \ref{41}) is still meaningful. Moreover, if $(A; R) \in \Cb$ then every finite subset of $A$ is contained a finite self-sufficient subset of $A$, so the pregeometry $PG(A; R)$ is a direct limit of finite $(C_i : i \in I)$ (for some directed set $I$) such that $C_i \unlhd C_j$ when $i < j$. In particular, if $A$ is countable, there are finite subsets $C_0 \unlhd C_1 \unlhd C_2 \unlhd \ldots $ of $A$ with $A = \bigcup_{i} C_i$. This raises the following question:

\begin{question}\label{59}\rm If $\A = (A, \cl)$ is a flat pregeometry, can $\A$ be written as the union of a directed system $(C_i : i \in I)$ of finite subpregeometries where $C_i \unlhd C_j$ when $i < j$?
\end{question} 

It is not hard to show that if $\A$ can be written in this way, then (using Theorem \ref{32}) it has a presentation: $\A = PG(A; R)$ for some $(A; R) \in  \Cb$.

\section{Free amalgamation and weak canonical bases}

Suppose $\A = (A, \cl)$ is a pregeometry with dimension function $d$. Let $B$ be a closed set in $\A$ and $\a$ a finite tuple of elements in $A$. The dimension of $\a$ over $B$, denoted by $d(\a/B)$, is the dimension of $\a$ in the localization of $\A$ over $B$ (ie. the contraction of $\A$ over $B$). If $B$ is finite dimensional, this is $d(\a B) - d(B)$ (where $\a B$ means the set consisting of the elements of $B$ and the elements of the tuple $\a$). 

We say that $\A$ has \textit{weak canonical bases over closed sets} if whenever $B$ is a closed set and $\a$ is a tuple of elements of $A$, then there is a closed $B_0 \subseteq B$ such that for every closed $B_1 \subseteq B$ we have $d(\a/B_1) = d(\a/B)$ iff $B_0 \subseteq B_1$. 

Such a property was considered for full algebraic matroids in \cite{DL}. We shall show:

\begin{thm}\label{61} If $\A$ is a strict gammoid, then $\A$ has weak canonical bases over closed sets. 
\end{thm}

For structures given by Hrushovski's construction, this sort of result is essentially folklore amongst model-theorists. But it seems worthwhile to include it here as it may be new to matroid theorists.

We shall use the fact from Section 2 that $\A$ has a presentation $\A = PG(A; R)$, and in the following, the predimension $\delta$ and the associated $\leq$ refer to this presentation.

We say that  $X, Y \subseteq A$ are \textit{freely amalgamated} over $X \cap Y$ if $R\vert (X\cup Y) = (R\vert X) \cup (R\vert Y)$. Note that this is equivalent to saying $\delta(X \cup Y) = \delta(X)+\delta(Y) - \delta(X\cap Y)$. We say that closed sets $X, Y$ in $A$ are \textit{independent} over $X \cap Y$ if $d(X \cup Y) = d(X)+d(Y) - d(X\cap Y)$. Note that this is equivalent to saying $d(X_1/X_2) = d(X_1/X_1\cap X_2)$.

\begin{lem}\label{62} Suppose $\A = (A; R) \in PG(\C)$ and $X, Y$ are closed in $\A$. Then  $X, Y$ are independent over $X \cap Y$ if and only if $X\cup Y \leq A$ and $X, Y$ are freely amalgamated over $X\cap Y$.
\end{lem}

\begin{proof}
Note that
$$d(X\cup Y) \leq \delta(X \cup Y) \leq \delta(X)+\delta(Y) - \delta(X\cap Y) = d(X)+d(Y)-d(X\cap Y)$$
(the last equality because $X$, $Y$, $X\cap Y$ are closed). The first inequality is an equality iff $X\cup Y \leq A$. The second inequality is an equality iff $X, Y$ are freely amalgamated over $X\cap Y$. Hence the result.
\end{proof}

A model theorist might refer to the following property as \textit{$CM$-triviality} of $\A$.

\begin{lem}\label{63} Suppose $\A = (A; R) \in PG(\C)$.
Suppose $A_1, A_2$ are closed in $\A$ and are independent over $A_1\cap A_2$.  Suppose $C$ is a closed in $\A$. Then $A_1\cap C$ and $A_2\cap C$ are independent over their intersection.
\end{lem}

\begin{proof} By Lemma \ref{61}, $A_1\cup A_2 \leq A$ and $A_1, A_2$ are freely amalgamated over their intersection. It follows that $(A_1\cap C)\cup (A_2\cap C) = (A_1\cup A_2)\cap C \leq C \leq A$ and $(A_1\cap C)$,  $(A_2\cap C)$ are freely amalgamated over their intersection. So the result follows by another application of Lemma \ref{61}.
\end{proof}

\begin{lem} \label{64} Suppose $\A = (A; R) \in PG(\C)$
If $B_1 \subseteq B$ are closed in $\A$ and $d(\a/B) = d(\a/B_1)$ then $\cl(\a B_1)\cap B = B_1$. 
\end{lem}

\begin{proof}
Let $X = \cl(\a B_1)\cap B$. As $B_1 \subseteq X \subseteq B$ we have 
\[d(\a X)-d(X) = d(\a/X) = d(\a/B) = d(\a/B_1) = d(\a B_1) - d(B_1).\]
Now, $\cl(\a X) = \cl(\a B_1)$, so we obtain $d(X) = d(B_1)$. So as $B_1$ and $X$ are closed, we have $X = B_1$.
\end{proof}

\medskip

\noindent\textit{Proof of Theorem \ref{61}.}
With the above notation, it will suffice to show that if $B_1, B_2$ are closed in $B$ with $d(\a/B_i) = d(\a/B)$ then $d(\a/B_1\cap B_2) = d(\a/B)$. Let $A_i = \cl(\a B_i)$ and $B_0 = B_1\cap B_2$.  By Lemma \ref{64}, $A_i \cap B = B_i$ for $i = 1, 2$.

Let $X = A_1\cap A_2 \supseteq \cl(\a B_0) \supseteq \cl(B_0)$. Now, $A_1$ and $B$ are independent over their intersection $B_1$, so by Lemma \ref{63}, $A_1 \cap A_2$ and $B \cap A_2$ are independent over their intersection. In other words, $X$ and $B_2$ are independent over $B_0$. 

Thus $d(\a/B) \leq d(\a/B_0) \leq d(X/B_0)  = d(X/B_2) = d(\a/B_2) = d(\a/B)$. We conclude that $d(\a/B_0) = d(\a/B)$ (and $X = \cl(\a B_0)$). \hfill $\Box$

\medskip

A similar style of argument gives the following.

\begin{thm}\label{65} Suppose $\A = (A; R) \in PG(\C)$ and $C \leq A$. If $F$ is a closed set  in the restriction $(C, \cl_C)$, denote by $\Ft$ its closure $\cl(F)$ in $\A$. If $F_1, F_2$ are closed in $(C, \cl_C)$, then $\Ft_1\cap \Ft_2 = \widetilde{F_1\cap F_2}$.
\end{thm}

\begin{proof}
Note that $\Ft_i \cap C = F_i$. As $C \leq A$ we have 
\[\delta(C) = \delta(\Ft_i \cup C) \leq \delta(\Ft_i)+\delta(C)-\delta(F_i) = \delta(C)\]
so we have equality and therefore $\Ft_i\cup C$ is the free amalgamation of $\Ft_i$ and $C$ over $F_i$. Let $Y = \Ft_1\cap \Ft_2$. It follows that $Y\cup C$ is the free amalgamation of $C$ and $Y$ over $F_1\cap F_2$. Thus as before  $\delta(C) = (C \cup Y) = \delta(C)+\delta(Y) - \delta(F_1\cap F_2)$. So $\delta(Y) = \delta(F_1\cap F_2)$ and it follows easily that $Y = \cl(F_1\cap F_2)$, as required.
\end{proof}

\begin{rem}\rm
It follows that, with the above notation, if $F_1,\ldots, F_r$ are closed sets in $(C,\cl_C)$, then $\cl(\bigcap_i F_i) = \bigcap_i \Ft_i$. Also, the dimension (computed in $C$) of $\bigcup_i F_i$ is equal to $d(\bigcup_i \Ft_i)$. Thus, if $\Delta^\A$ and $\Delta^C$ denote the function $\Delta$ from Definition \ref{41} in $\A$ and $(C, \cl_C)$ respectively, then 
\[\Delta^C(F_1,\ldots , F_r) = \Delta^\A(\Ft_1,\ldots, \Ft_r).\]
Applying Remarks \ref{44} to this equation we obtain:
\[\delta(\bigcup_i \Ft_i)-d(\bigcup_i \Ft_i)+ \rho(\Ft_1,\ldots, \Ft_r) =
\delta(\bigcup_i F_i)-d(\bigcup_i F_i)+ \rho(F_1,\ldots, F_r).\]
The second terms of each side are equal. As $C \leq A$ we have $\delta(\bigcup_i \Ft_i) \geq \delta(\bigcup_i F_i)$; as $(\bigcup_i \Ft_i) \cap C = \bigcup_i F_i$ we have $\rho(\Ft_1,\ldots, \Ft_r) \geq \rho(F_1,\ldots, F_r)$. Thus we have equality in both of these, and we can view $\bigcup_i \Ft_i$ as being `freely constructed' from the $\Ft_i$ (over $\bigcup_i F_i$). Further results on freeness in flat pregeometries (which do not depend on their characterization in terms of Hrushovski constructions) can be found in Holland's paper \cite{Holland}
\end{rem}

Suppose $(B_1; R_1), (B_2; R_2) \in \C$ (or even in $\Cb$) and $A \leq (B_i; R_i)$ (with $R_1\vert A = R_2\vert A$). We can assume $A = B_1\cap B_2$ and let $C = B_1\cup B_2$.  Consider $C$ as a set system with relations $R = R_1 \cup R_2$. Then it can be shown that $B_i \leq (C; R) \in \Cb$ and of course, $B_1, B_2$ are freely amalgamated over $A$ in $C$. So we refer to  $C$ as the free amalgam of $B_1$ and $B_2$ over $A$.

We note  that the pregeometry on $C$ in this free amalgam does not depend on the presentations $R_1, R_2$ of the pregeometries on $B_1$ and $B_2$. In fact, in the terminology of (\cite{Oxley}, 12.4), $C$ is \textit{the free amalgam} of the pregeometries $B_1$ and $B_2$ over $A$: if $C'$ is any other pregeometry on the set $C$ (the disjoint union of $B_1, B_2$ over $A$) which extends the given pregeometries on $B_1$ and $B_2$, then any set which is independent in $C'$ is independent in $C$. This can be seen by combining (\cite{Oxley}, 12.4.3 and 12.4.4) and Lemma \ref{62} here, but we summarise thse arguments as:

\begin{lem}\label{67} With the above notation, for $X \subseteq C$ let $d(X_{B_i})$ denote the dimension in $B_i$ of $X \cap B_i$ and likewise $d(X_A)$. Define $\eta(X) = d(X_{B_1})+d(X_{B_2}) - d(X_A)$ and $\zeta(X) = \min(\eta(Y) : X \subseteq Y \subseteq C)$. Then the dimension $d_C(X)$ of $X$ in $(C; R_1 \cup R_2)$ is equal to $\zeta(X)$. Moreover, if $C'$ is as above with dimension function $d_{C'}$ then $d_{C'}(X) \leq \zeta(X)$.
\end{lem}

\begin{proof} By submodularity of $d_{C'}$, if $Y \subseteq C$ then 
\[d_{C'}(Y) \leq d_{C'}(Y_{B_1})+d_{C'}(Y_{B_2}) - d_{C'}(Y_A) = \eta(Y).\]
Thus for $X \subseteq Y$ we have $d_{C'}(X) \leq d_{C'}(Y) \leq \eta(Y)$, so $d_{C'}(X) \leq \zeta(X)$. In particular, this is true with $C' = C$.

On the other hand, let $X \subseteq C$ and $Z$ its closure in $C$. Then $d_C(X) = d_C(Z)$ and by Lemma \ref{62} this is equal to $\eta(Z)$. By definition,  $\eta(Z) \geq \zeta(Z) \geq \zeta(X)$, so the result follows.
\end{proof}

\section{Questions and further observations}

\subsection{A classification of homogeneous flat geometries?}
The main purpose of Hrushovski's construction in \cite{Hr} was to provide  model theorists with examples of strongly minimal sets (or more generally, regular types) whose geometries were significantly different from those that were previoulsy known. Though the construction is by now very familiar, the class of geometries which arises has still seemed  to be somewhat mysterious. As a test problem, we might ask for a description of a class of geometries which includes the ones in \cite{Hr}, and a classification of them.

Until recently, it looked as though this was a hopeless problem for at least two reasons:
\begin{enumerate}
\item[(A)]  the construction appears to produce a wide variety of non-isomorphic countable dimensional geometries;
\item[(B)] one of the few properties we know about the geometries, namely flatness, does not look to be very convenient to use.
\end{enumerate}

However,  \cite{MFDE1} and \cite{MFDE2} (and \cite{MFThesis}) show that there are only countably many \textit{local} isomorphism types of countable dimensional geometries (of strongly minimal sets) produced in \cite{Hr}: by this we mean the isomorphism type of the  geometry obtained after localizing  over a finite set so that lines have infinitely many points. See the end of this subsection for a more precise statement. Furthermore, the point of these notes is to correct the impression in (B): Mason's Theorem shows us how flatness implies that the geometry arises from a predimension.

The (pre)geometry of a regular type (or a strongly minimal set) has a large automorphism group. The geometry is \textit{homogeneous}, meaning that the pointwise stabiliser of a finite dimensional flat acts transitively on the set of points outside the flat. So we might ask:

\begin{question}\label{71} \rm Is there a classification of the countable, infinite dimensional homogeneous flat geometries with  infinitely many points on a  line?
\end{question}

Possibly we should also include as a hypothesis here that our geometry is a direct limit of finite flat geometries, as in Question \ref{59}. 

However, this may still be too optimistic, so we give a weaker version of the question in terms of amalgamation classes of finite flat geometries.

Recall the notation $\P_k = (PG(\C_k), \unlhd_k)$ introduced at the end of Section 5. This is an \textit{amalgamation class} (\cite{MFDE1}, 6.4). By this we mean that if $B \unlhd_k A \in \P_k$ then $B \in \P_k$, and if $f_i : A \to B_i \in \P_k$ are $\unlhd_k$-embeddings (meaning $f_i(A) \unlhd_k B_i$) for $i = 1,2$, then there is $C \in \P_k$ and $\unlhd_k$-embeddings $g_i : B_i \to C$ with $g_1\circ f_1 = g_2\circ f_2$. Indeed, we can assume that the $f_i$ are inclusions and there are presentations $(B_i; R_i) \in \C_k$ of the $B_i$ with $A \leq (B_i; R_i)$. By (\cite{MFDE1}, 4.2) we may assume that $R_1$ and $R_2$ agree on $A$: we can replace $R_2\vert A$ by $R_1 \vert A $ in $R_2$ without changing the pregeometry on $B_2$. So then we can take $C \in \P_k$ to be the pregeometry of the free amalgam of $(B_1; R_1)$ and $(B_2; R_2)$ over $A$, as at the end of Section 6.

Suppose $(\H, \unlhd_k) \subseteq \P_k$ is also an amalgamation class. Then there is a pregeometry $P$ which is the union of finite subgeometries $P_i$ such that:
\begin{enumerate}
\item[(1)]  $P_1 \unlhd_k P_2 \unlhd_k P_3 \unlhd_k \ldots $;
\item[(2)] if $A \unlhd_k P_i$ and $A \unlhd_k B \in \H$ then there is $j\geq i$ and a $\unlhd_k$-embedding $h : B \to P_j$ with $h(a) = a$ for all $a \in A$.
\end{enumerate}
It can be shown that $P$ is determined up to isomorphism by $\H$ and we refer to it as the \textit{generic structure} associated with $\H$. Write $A \unlhd_k P$ to mean that $A \unlhd_k P_j$ for some $j$. If $A_1, A_2 \unlhd_k P$ and $\gamma : A_1 \to A_2$ is an isomorphism of pregeometries, there is an automorphism $f$ of $P$ which extends $\gamma$ (and has the property that for every $m$ there is $n \geq m$ such that $f(P_m) \unlhd_k P_n$, and therefore $f$ preserves $\unlhd_k$). We refer to this property as \textit{$\unlhd_k$-homogeneity}. Note that this differs from what one might expect as \textit{a priori} the set of pairs of substructures to which it applies appears to depend on the chain used in (1): see Section 6 of \cite{MFDE1} for further comments on this. In any case, we have the following:

\begin{lem} Suppose $\H \subseteq \P_k$ is an amalgamation class and $P$ is the generic structure assocated to $\H$. Then $P$ is a homogeneous pregeometry.
\end{lem}

\begin{proof} Let $P_i$ be as in (1), (2) above. Suppose $A \subseteq P$ is finite and $b_1, b_2 \in P \setminus \cl(A)$. Take $i$ with $A \unlhd_k P_i$ and $b_1, b_2 \in P_i$. There is a presentation $(P_i; R) \in \C_k$ of $P_i$ with $A \leq P_i$ and as $b_1, b_2 \not\in \cl(A)$ it follows that $A\cup \{b_i\} \leq (P_i; R)$ and  $b_i$ is not contained in a relation in $R\vert A\cup \{b_i\}$. Thus the induced set-systems on the $A \cup \{b_i\}$ are isomorphic over $A$, and the same is therefore true of the pregeometries. It follows by $\unlhd_k$-homogeneity that there is an automorphism of $P$ which fixes all of $A$ and takes $b_1$ to $b_2$ (and which preserves $\unlhd_k$ with respect to the $P_i$). 

Finally we claim that there is an automorphism of $P$ which fixes all of $\cl(A)$ and takes $b_1$ to $b_2$. An argument similar to that given on pp. 96--97 of \cite{HDM} can be used.
\end{proof}

So the following is relevant to Question \ref{71}:

\begin{question}\rm What are the subclasses $\H$ of $\P_k = (PG(\C_k), \unlhd_k)$ which are amalgamation classes?
\end{question}

Here we might want to assume that there is no bound on the dimension or number of points on a line of the pregeometries in $\H$. We might also want to assume that $\H$ contains a $k$-cycle. Then we might optimistically conjecture that $\H = \P_k$. As a first step towards this question, one might consider those $\H$ which are closed under \textit{free} amalgamation.

\medskip

We review briefly some of the results in \cite{MFDE1} and \cite{MFDE2} using this terminology. Let $G_k$ be the generic structure for $\P_k$. Theorems 6.9 and 4.3 of \cite{MFDE1} show that these pregeometries are non-isomorphic for different values of $k$. Theorem 5.5 of \cite{MFDE1} shows that $G_k$ is isomorphic to its localization over any finite subset. The results of \cite{MFDE2} can be phrased as saying that for each of the countable, infinite dimensional homogeneous pregeometries constructed in \cite{Hr} there is a $k$ such that a some localization over a finite set  is isomorphic to $G_k$. 

\medskip

A rather more subtle type of question is:

\begin{question} \rm Formulate a natural model theoretic / combinatorial conjecture (or result) which has flatness as part of its conclusion.
\end{question}

\subsection{Gammoids}

A \textit{gammoid} is a submatroid of a strict gammoid, ie. a matroid which can be embedded in a matroid of the form $PG(A; R)$ for $(A; R) \in \C$. The following is a well-known open problem in matroid theory (see \cite{Oxley}, Problem 14.7.1).

\begin{question}\rm
Is there an algorithm to test whether or not a given matroid is a gammoid?
\end{question}

Note that Corollary 3.9 of \cite{MFDE1} shows that class of submatroids of matroids in $PG(\C_k)$ is independent of $k$ (for $k \geq 3$). 

\medskip

Somewhat speculatively, we ask:

\begin{question}\rm 
If a gammoid has weak canonical bases over closed sets (as in Section 6), is it a strict gammoid?
\end{question}

\end{document}